\documentclass[12pt]{article}

\usepackage{amssymb}
\usepackage{amsmath}
\usepackage{amsbsy}
\usepackage{amscd}
\usepackage{amsfonts}
\usepackage{amsthm}
\usepackage{mathrsfs}
\usepackage{verbatim}
\usepackage[colorlinks]{hyperref}
\usepackage{fullpage}
\usepackage{mathdots}
\usepackage{graphicx,subfigure}
\usepackage[english]{babel}
\usepackage[utf8]{inputenc}
\usepackage{tikz}

\usepackage{authblk}

\usepackage{mathtext}

\usetikzlibrary{backgrounds,fit, matrix}
\usetikzlibrary{positioning}
\usetikzlibrary{calc,through,chains}
\usetikzlibrary{arrows,shapes,snakes,automata, petri}

\newtheorem{Theorem}[equation]{Theorem}

\newtheorem{Lemma}[equation]{Lemma}

\theoremstyle{definition}

\newtheorem{Remark}[equation]{Remark}

\numberwithin{equation}{section}
\numberwithin{figure}{section}

\newcommand{\PP}{{\mathbb P}}
\newcommand{\C}{{\mathbb C}}
\newcommand{\Z}{{\mathbb Z}}
\newcommand{\Q}{{\mathbb Q}}

\newcommand{\N}{{\mathbb N}}
\newcommand{\A}{{\mathbb A}}

\newcommand{\mb}[1]{\mathbb{ #1}}

\newcommand{\mc}[1]{\mathcal{#1}}
\newcommand{\ms}[1]{\mathscr{#1}}

\newcommand{\mt}[1]{\text{#1}}

\def\ds#1{\displaystyle{#1}}

\begin{document}

\title{On the cohomology rings of Grassmann varieties and Hilbert schemes}


\author[1]{Mahir Bilen Can}
\author[2]{Jeff Remmel}

\affil[1]{{mahirbilencan@gmail.com}}
\affil[2]{{jremmel@ucsd.edu}}

\date{July 20, 2017}
\maketitle

\begin{abstract}
By using vector field techniques, we compute the ordinary and equivariant cohomology rings of Hilbert scheme of points in 
the projective plane in relation with that of a Grassmann variety. 

\vspace{.5cm}
\end{abstract}

\normalsize

\section{Introduction}\label{S:Introduction}


The Hilbert scheme of points in the plane is a useful moduli that brings forth the power of algebraic geometry 
for solving combinatorial problems not only in commutative algebra but also in representation theory such as 
Garsia-Haiman modules and Macdonald polynomials~\cite{MR1839919}. Its geometry has deep connections 
with physics. For example,  by considering the cohomology rings of all Hilbert schemes of points together, one obtains Fock 
representation of the infinite dimensional Heisenberg algebra, which has importance for string theorists (see
~\cite{MR1711344} and the references therein).

We denote by $Hilb_k(X)$ the Hilbert scheme of $k$ points in a complex algebraic variety $X$. In this paper, 
we will be concerned with the cohomology ring of $Hilb_k(\PP^2)$. Its additive structure was first described 
by Ellingsrud and Str{\o}mme in~\cite{MR1228610}, where it was shown that the Chern characters of 
tautological bundles on $Hilb_k(\PP^2)$ are enough to generate it as a $\Z$-algebra. In relation with the 
Virasoro algebra, a finer system of generators for the cohomology ring is described in~\cite{MR1931760} by 
Li, Qin, and Wang. Among other things, the same authors showed in~\cite{MR1857595} that the equivariant 
homology ring of Hilbert scheme of points in the plane is generated by the Jack symmetric functions. Around 
the same time, the $\C^*$-equivariant homology ring of the Hilbert scheme of points in $\C^2$ (with $\C$ as 
the group of coefficients) is described by Vasserot in~\cite{MR1805619}. In~\cite{MR1865244}, it is shown 
by Lehn and Sorger that the cup product on $H^*(Hilb_k(\C^2),\Q)$ is equal to the convolution product on 
the center of the group ring of the symmetric group. 

The computation that is closest to ours in spirit is that of Evain \cite{MR2323683}, where 
he used Brion's Localization Theorem for torus equivariant (Edidin-Graham-)Chow rings. In general, Chow
rings are not the same as cohomology rings, however, for a nonsingular complex projective variety $X$ on 
which an algebraic torus acts with finitely many points, the cycle map gives a graded algebra isomorphism
between the torus equivariant Chow ring and the equivariant cohomology ring of $X$. (This is easy to see 
from \cite[Corollary 3.2.2]{MR1466694}.) Therefore, Evain's results are applicable in a broader setting than 
ours. A  major difference between our approach is that we 
utilize Gotzmann's embedding of the Hilbert scheme into a single Grassmann variety which allows us to use 
the zero schemes of equivariant vector fields. Thus our calculations 
are not the same as Evain's calculations.

We will consider several different vector fields which are defined by the flows of one-parameter subgroups 
(1-PSG's) of the special linear group that acts on the grassmannian. Let $B(2)$ denote the Borel subgroup 
consisting of upper triangular matrices in $SL(2)$. Two of the one-parameter subgroups we consider are 
the unipotent radical and the maximal torus of $B(2)$. A variety is called a regular $B(2)$-variety if it admits 
an action of $B(2)$ with a unique fixed point. It turns out that all partial flag varieties, and hence all grassmannians,  
are examples to regular $B(2)$-varieties. This theory is particularly well-suited for studying Schubert 
subvarieties and  the like as well. See Aky{\i}d{\i}z and Carrell's work in~\cite{MR998938}. The cohomology ring 
of a $B(2)$-regular variety admits an $SL(2)$-module structure and all of this is computable rather easily via 
torus weights on the tangent space at the unique fixed point. We do this computation for the Grassmann 
variety though we do not claim originality (see~\cite{MR0460734}). Nevertheless, by using our computation 
along with a work of Brion and Carrell (\cite{MR2043405}), we obtain a very concrete description of the 
equivariant cohomology ring of grassmannians. 

In the examples of grassmannians we looked at, the vector fields that are defined by the 1-PSGs of $B(2)$ 
are not tangential to the (embedded copy of) Hilbert scheme of points. Furthermore, when $k=3$, we show 
that $Hilb_k(\PP^2)$ is not a regular $B(2)$-variety. Since vector fields restrict locally on affine subsets, by 
inspecting the tables of Betti numbers of $Hilb_k(\PP^2)$'s, we conjecture that there is no $k>1$ such that 
the Hilbert scheme of $k$ points on a two-fold has a regular $B(2)$-variety structure.

The third one-parameter subgroup action we consider is more natural from an ideal theoretic point of view. 
It was used by Gotzmann in~\cite{MR968319} to describe a useful stratification for $Hilb_k(\PP^2)$. We use 
this 1-PSG to compute cohomology rings of both of the grassmannians and Hilbert schemes. Although it is
not as concrete as in the case of a regular $B(2)$-variety, the 1-PSG we consider leads to a description of
the $\C^*$-equivariant cohomology ring of $Hilb_k(\PP^2)$ as well. 

Now we are ready to describe the contents of our paper. In Section~\ref{S:Preliminaries} we review some of 
the foundational results of Aky{\i}d{\i}z, Carrell, and Lieberman on filtrations associated with the vector fields. 
In Section~\ref{S:Appendix}, by using a theorem of Brion and Carrell we present a simple description of the 
$B(2)$-equivariant cohomology ring of the Grassmann variety. 
In Section~\ref{S:GaGm variety} we construct the one-parameter subgroup of $GL(3)$ that is fundamental to
our computations and observe that the resulting vector field has finitely many fixed points on the Grassmann
variety and the Gotzmann embedding of the Hilbert scheme is equivariant with respect to this action. In the 
subsequent Section~\ref{S:Carrell}, we carry out Carrell's analysis of the Chern classes of grassmannians for
the Hilbert scheme of points. This analysis shows that we can apply the restriction theorem that is explained in Section~\ref{S:Preliminaries}.
In particular, we prove, the well-known (by other methods) result that the ring $H^*(Hilb_k(\PP^2),\C)$ has a 
basis consisting of Schur polynomials. In Section~\ref{S:ECoh}, we prove that the cohomology rings of both 
of the Grassmann variety and the Hilbert scheme are generated by the equivariant Chern classes and there 
is a $\C^*$-equivariant surjection $H^*_{\C^*}(Gr(n-k,V),\C)\rightarrow H^*_{\C^*}(Hilb_k(\PP^2),\C)$.

\vspace{1cm}
\textbf{ Acknowledgements.} 
We started to work on this problem during Adriano Garsia's Seminar on Diagonal Harmonics, which took place 
in Spring Quarter Semester of 2017 at the UCSD. We thank Adriano Garsia, Marino Romero, Dun Qui, and to 
Guoce Xin. We thank Ersan Aky{\i}ld{\i}z, Soumya Dipta Banerjee, Jim Carrell, \"Ozg\"ur Ki\c{s}isel for helpful 
discussions, pointers to the literature, and for encouragements. Finally, we are grateful to Kiumars Kaveh for 
his extremely careful reading and the critique of the first version of this manuscript.

\section{Preliminaries} \label{S:Preliminaries}
Throughout our paper by a variety we mean a reduced and irreducible scheme defined over $\C$, the field of 
complex numbers.

\subsection{Vector fields with isolated zeros.}
Let $X$ be a nonsingular complex projective variety, $L$ be a holomorphic vector field on $X$, and $Z$ denote 
the subscheme of $X$ that is defined by the sheaf of ideals $i(L) \Omega^1_X \subseteq \mc{O}_X$, where $i
(L):\Omega^p_X\rightarrow \Omega^{p-1}_X$ is the contraction operator (also known as the interior product) 
defined by the vector field on the sheaf of germs of holomorphic $p$ forms on $X$. Since $i(L)^2=0$ and 
$\mc{O}_X = \Omega^0_X$, there is a complex of sheaves of holomorphic forms 
\begin{align}\label{A:complex of sheaves}
0 \rightarrow \Omega^n_X \rightarrow \Omega^{n-1}_X \rightarrow \cdots \rightarrow \Omega^1_X \rightarrow 
\mc{O}_X \rightarrow 0 
\end{align}
where $n=\dim X$. It is shown in~\cite{MR0435456} by Carrell and Lieberman that if $Z$ is finite but nontrivial, 
then the ring of functions $H^0(Z)=H^0(Z,\mc{O}_Z)$ on $Z$ has a decreasing filtration 
$$
H^0(Z,\mc{O}_Z) =F_{-n} \supset F_{-n+1}\supset \cdots \supset F_0=0
$$ 
such that 
\begin{enumerate}
\item $F_i F_j \subset F_{i+j}$, 
\item $\oplus F_{-i}/F_{-i+1} \cong \oplus H^i ( X,\C)$ as graded algebras, and
\item the odd cohomology groups of $X$ vanish.
\end{enumerate}
It is possible to adopt, to a certain degree, this technique of vector fields to singular subvarieties of $X$. Since 
this is pertinent to the goals of our paper, we briefly review relevant results from \cite{MR827353} to provide the details to make the above 
discussion clearer.

Associated with the complex (\ref{A:complex of sheaves}) is the hypercohomology ring $\underline{H}_X^*$ of 
$X$ where there is a spectral sequence with $E^{-p,q}_1= H^p ( X, \Omega^q_X)$ 
that gives rise to the aforementioned 
filtration. The inclusion map $Z\hookrightarrow X$ induces an isomorphism between hypercohomologies of 
$Z$ and $X$ associated with the sheaf complex above (see ~\cite[Lemma 1]{MR827353} and the paragraph 
following it). Now, suppose we have an algebraic action of a 1 dimensional torus, $\lambda : \C^* \times X  
\rightarrow X$ such that the differential of the flow of $\lambda$ is $L$. It follows that $X^\lambda$, the fixed 
point subscheme of $\lambda$,  is equal to $Z$. Let $Y\subset X$ be a possibly singular subvariety and assume 
that $Y$ is stable under $\lambda$. Let $i$ and $j$ denote the inclusions $i: Y\hookrightarrow X$ and $j:Y\cap 
Z\hookrightarrow Z$, respectively, and $Z'$ denote $Z\cap Y$. 
Then we let $\phi$ be the composition of the following 
maps:
$$
\underline{H}^0_Z = \sum_{p\geq 0} H^p(Z; \Omega_Z^p) \rightarrow H^0(Z) \rightarrow H^0(Y\cap Z).
$$
Under the assumption that the induced map $j^* : H^0 (Z) \rightarrow H^0(Y\cap Z) $ is surjective,  Aky{\i}ld{\i}z, Carrell, and Lieberman show 
in~\cite[Theorem 3]{MR827353} that there exists a commuting diagram of graded 
algebra homomorphisms as in Figure~\ref{F:Restriction}, where $\phi'$ is the homomorphisms determined 
by $\phi$. 
\begin{figure}[htp]
\centering
\begin{tikzpicture}[scale=1]
\begin{scope}
\node (a1) at (-2,2) {$\textbf{gr}\ \underline{H}_Z^0$} ;
\node (a2) at (-2,0) {$\textbf{gr}\ H^0(Y\cap Z)$} ;
\node (b1) at (2,2) {$\sum_{p\geq 0} H^p(X; \Omega_X^p)$} ;
\node (b2) at (2,0) {$H^*(Y)$} ;
\node at (0.25,0.25) {$\psi$};
\node at (-2.25,1) {$\phi'$};
\node at (2.25,1) {$i^*$};
\draw[->, thick] (a1) to (a2);
\draw[->, thick] (a1) to (b1);
\draw[->, thick] (b1) to (b2);
\draw[->, thick] (a2) to (b2);
\end{scope}
\end{tikzpicture}
\caption{Restriction to a subvariety}
\label{F:Restriction}
\end{figure}
Moreover, it is shown that $\phi'$ is surjective so that the image of $\psi$ is $\sum_{p\geq 0}i^*H^p( X; \Omega_X^p
)$. If all odd Betti numbers of $Y$ vanish, then $\psi$ is an isomorphism if and only if either $\psi$ is injective or 
surjective. In particular, $\psi$ is an isomorphism if and only if $\sum_{p\geq 0}i^*H^p(X;\Omega_X^p)=H^*(Y)$.

\begin{Remark}\label{R:importance of}
In our case both of $Y$ and $Z$ will be nonsingular and we will have $H^q(Y;\Omega_Y^p)=0$ for all $p\neq q$.
Therefore, the diagram in Figure~\ref{F:Restriction} will simplify as in Figure~\ref{F:simplified}.
\begin{figure}[htp]
\centering
\begin{tikzpicture}[scale=1]
\begin{scope}
\node (a1) at (-2,2) {$\textbf{gr}\ H^0(Z)$} ;
\node (a2) at (-2,0) {$\textbf{gr}\ H^0(Y\cap Z)$} ;
\node (b1) at (2,2) {$H^*(X)$} ;
\node (b2) at (2,0) {$H^*(Y)$} ;
\node at (0.25,0.25) {$\psi$};
\node at (-2.25,1) {$\phi'$};
\node at (2.25,1) {$i^*$};
\draw[->, thick] (a1) to (a2);
\draw[->, thick] (a1) to (b1);
\draw[->, thick] (b1) to (b2);
\draw[->, thick] (a2) to (b2);
\end{scope}
\end{tikzpicture}
\caption{$Y$ is nonsingular, $Z$ is finite.}
\label{F:simplified}
\end{figure}
\\
(See diagram (2.2) in \cite{MR827353}.)
\end{Remark}

\subsection{$B(2)$-regular varieties.}\label{SS:B2}
In this subsection we will briefly review two important results on the structure of nonsingular projective varieties 
that admit a solvable group action. We assume $X$ is a nonsingular projective variety (over an algebraically 
closed field). Let $T$ be a torus acting algebraically on $X$ and we assume that its fixed point subscheme 
$Z$ is finite. We pick a one-parameter subgroup $\lambda$ with $X^{\lambda} = X^T$. For $p\in X^{\lambda}$ 
define the sets $C_p^+ = \{y \in X:\ \ds{\lim_{t \to 0} t \cdot y = p,}\ t \in \lambda \}$ and $C_p^- = \{y \in X:\ \ds
{\lim_{t \to \infty} t \cdot y = p,}\ t \in \lambda \}$, called the plus cell and minus cell of $p$, respectively. The 
following results are proven by Bia{\l}ynicki-Birula in~\cite{MR0366940,MR0453766}:
\begin{enumerate}
\item $C_p^+$ and $C_p^-$ are locally closed subvarieties isomorphic to affine space;
\item if $T_p X$ is the tangent space of $X$ at $p$, then $C_p^+$ (resp., $C_p^-$) is $\lambda$-equivariantly 
isomorphic to the subspace $T_p^+ X$ (resp., $T_p^- X$) of $T_p X$ spanned by the positive (resp., 
negative) weight spaces of the action of $\lambda$ on $T_p X$.
\end{enumerate}
Consequently, there exists a filtration 
$$
X^{\lambda}  = V_0 \subset V_1 \subset \cdots \subset V_n = X \qquad (n = \dim X),
$$
of closed subsets such that for each $i=1,\dots,n$, $V_i - V_{i-1}$ is the disjoint union of the plus (respectively  
minus) cells in $X$ of complex dimension $i$. It follows that the odd-dimensional integral cohomology groups 
of $X$ are trivial, the even-dimensional integral cohomology groups of $X$ are free, and the Poincar\'e 
polynomial $P_X(t) := \sum_{i=0}^{2n} \dim_{\C} H^{i}(X, \C) t^i$ of $X$ is given by
$$
P_X(t) = \sum_{p \in X^{\lambda}} t^{2 \dim C_p^+} = \sum_{p \in X^{\lambda}} t^{2 \dim C_p^-}.
$$
Since the odd-dimensional cohomology groups are trivial, it is convenient to focus on $q$-Poincar\'e 
polynomial obtained from $P_X(t)$ by the substitution $t^2 = q$.

Let $B(2,\C)$ denote the Borel subgroup of upper triangular matrices in $SL(2,\C)$ and let $\mb{G}_m$ and 
$\mb{G}_a$ denote, respectively, the maximal torus and unipotent radical in $B(2,\C)$,
\begin{align*}
\mb{G}_m = \left\{ \begin{pmatrix} t & 0 \\ 0 & t^{-1} \end{pmatrix} \ : t\in \C^* \right\} \ \text{ and } \ 
\mb{G}_a = \left\{ \begin{pmatrix} 1 & z \\ 0 & 1 \end{pmatrix} \ : z\in \C \right\}.
\end{align*}
Let $X$ be a nonsingular projective variety admitting actions 
\begin{align*}
\lambda:\ \mb{G}_m \times X \rightarrow X \ \text{ and } \ \varphi: \mb{G}_a\times X \rightarrow X 
\end{align*}
satisfying the following properties: 
\begin{enumerate}
\item The fixed point variety of the action of $\varphi$ is a singleton, $X^\varphi := \{ s_0\}$. 
\item There is an integer $p\geq 1$ such that $\lambda (t) \varphi(z) \lambda(t^{-1}) = \varphi( t^pz)$ for all 
$t\in \C^*$ and $z\in \C$.
\end{enumerate}
In this case, $X$ is called a $(\mb{G}_a,\mb{G}_m)$-variety. It turns out that the fixed point set of $\mb{G}_m$ 
action, $X^{\mb{G}_m}$, on a $(\mb{G}_a,\mb{G}_m)$-variety is always finite and includes the point $s_0$, 
\cite{MR998938}. We set $X^{\mb{G}_m} := \{ s_0,s_1,\dots, s_r\}$. The minus-cell $C_{s_0}^-$ corresponding 
to $s_0$ is open in $X$, hence, the $\mb{G}_m$-weights on the tangent space $T_{s_0} X$ are all negative. 
We fix a $\mb{G}_m$-eigenbasis $\{e_1,\dots, e_n\}$ for $T_{s_0}X$, let $a_1,\dots, a_n$ denote the 
$\mb{G}_m$-weights of the basis vectors. Finally, let $x_1,\dots, x_n$ denote the corresponding dual basis so 
that the coordinate ring of $C_{s_0}^- \simeq T_{s_0} X$ is of the form $\C[C_{s_0}^-] = \C [x_1,\dots, x_n]$. 
Let us denote this ring by $A(C_{s_0}^-)$. Then $A(C_{s_0}^-)$ is an $\N$-graded algebra with the principal 
grading given by $\deg x_i = - a_i$ for $i=1,\dots, n$.

The zero scheme $Z_a$ of the vector field $N_a$ defined by $N_a=\left.\frac{ d \varphi }{ d z } \right\vert_{z=0}$ 
has $\{s_0\}$ as its support. Then viewed as derivation, $N_a$ acts on $\C[x_1,\dots, x_n]$, the symmetric 
algebra on the cotangent space at $s_0$. Let $\phi_i$ denote the polynomial $N_a(x_i)\in \C[x_1,\dots, x_n]$.

In \cite{MR998938}, Aky{\i}ld{\i}z and Carrell proved that, for $i=1,\dots, n$, the polynomials $\phi_i$ are 
homogenous with $\deg \phi_i = p-a_i$. Moreover, $\phi_1,\dots, \phi_n$ form a regular sequence and 
the coordinate ring $A(Z_a)$ of $Z_a$ is the quotient ring $A(C_{s_0}^-)/I(Z_a)$, where $I(Z_a)$ is the 
ideal generated by $\phi_i$'s. Consequently, the Poincar\'e polynomial of $Z_a$ is given by 
\begin{align}\label{A:roots of unity}
P_{Z_a}(q) = \prod_{i = 1}^n \frac{1 - q^{p-a_i}}{1 - q^{-a_i}}.
\end{align}
Furthermore, there exists a graded algebra isomorphism $\Phi : A(C_{s_0}^-)\rightarrow H^*(X, \C)$ such that 
$$
\Phi ( A(C_{s_0}^-)_{ip} ) = H^{2i} ( X, \C).
$$

\begin{Remark}\label{A:criteria}
It follows from (\ref{A:roots of unity}) that if $z_0$ is a root of the Poincar\'e polynomial of a $(\mb{G}_a,
\mb{G}_m)$-variety, then $z_0$ has to be a root of unity. This provides a criteria for deciding when $X$ 
is not a $(\mb{G}_a,\mb{G}_m)$-variety.
\end{Remark}

\section{The cohomology ring of grassmannians}\label{S:Appendix}

\subsection{Ordinary cohomology ring of $Gr(n-k,V)$ as a $B(2)$-variety.}\label{S:Regular grassmannian}

Let $e_1,\dots, e_n$ be an ordered basis for the finite dimensional vector space $V$. The general linear group 
$G=GL(V)$ acts on $V$, hence it acts on $Gr(n-k,V)$. Both of these actions are transitive. For $\tau\subset\{1,
\dots,n\}$ with $k$ elements, let $W_\tau$ denote the span of vectors $e_{i_j}$ with $i_j \in \tau$, $j=1,\dots, 
n-k$. For simplicity, when $\tau=\{1,\dots,n-k\}$ we denote $W_\tau$ by $W_0$.

The stabilizer subgroup $P:=\mt{Stab}_{G} (W_0)$ of $W_0=\mt{span}_\C\{e_1,\dots, e_{n-k}\}$ is a maximal 
parabolic subgroup consisting of matrices of the form $\begin{pmatrix} A & B \\ \mathbf{0} & C \end{pmatrix}$, 
where $A$ is an $n-k \times n-k$ invertible matrix and $C$ is a $k\times k$ invertible matrix. Thus, 
$$
Gr(n-k,V) \cong G/P.
$$ 
Let $\mathbf{n} \in End(V)$ denote the ``regular'' nilpotent matrix whose $(i,j)$-th entry is 1 if $j=i+1$ and 0 
otherwise. There is a corresponding one-parameter subgroup $\mb{G}_a:=\mt{exp}(z\mathbf{n})$ ($z\in \C$) 
in $GL(V)$. Note that 
\begin{align}\label{A:nilp action}
\mt{exp}(z \mathbf{n}) \cdot e_i = e_i + z e_{i-1} + \frac{z^2}{2!} e_{i-2} + \cdots + \frac{z^{i-1}}{i!} e_1.
\end{align}
It is now easy to see from (\ref{A:nilp action}) that $W_0$ is the unique fixed point of $\mb{G}_a$-action.

By Jacobson-Morozov Theorem, we know that there exists an $SL(2)$-triple in $GL(V)$ such that $\mb{G}_a$ 
is the maximal unipotent subgroup of $SL(2)$, see Section 3.2 in McGovern's article in~\cite{MR1925828}. 
The appropriate one dimensional torus forming the $SL(2)$-triple along with $\varphi(z)=\mt{exp}(z\mathbf{n}) 
\in SL(V)$ is the diagonal torus 
$$
\lambda(t)= \mt{diag}(t^{n-1},t^{n-3},\dots, t^{-n+3},t^{-n+1}).
$$ 
Since
\begin{align*}
\lambda ( t) \varphi (z) \lambda(t^{-1}) \cdot e_i &= \lambda ( t) \varphi (z)  \cdot  t^{-(n-(2i-1))} e_i \\
&= \lambda ( t)  t^{-(n-(2i-1))}  ( e_i + z e_{i-1} + \frac{z^2}{2!} e_{i_2} + \cdots + \frac{z^i}{i!} e_1 )\\
&= t^{-(n-(2i-1))}  ( t^{(n-(2i-1))} e_i + t^{n-(2i-3)} z e_{i-1} + t^{n-(2i-5)} \frac{z^2}{2!} e_{i_2} + \cdots + t^{n-1} 
\frac{z^i}{i!} e_1 )\\
&=  e_i + t^{2} z e_{i-1} + t^{4} \frac{z^2}{2!} e_{i-2} + \cdots + t^{2i-2} \frac{z^i}{i!} e_1 \\
&= \varphi(t^2 z)\cdot e_i,
\end{align*}
$Gr(n-k,V)$ is a $(\mb{G}_a,\mb{G}_m)$-variety.

The $k(n-k)$-dimensional basic open affine neighborhood of $W_0$ is $O_{\tau_0}= \{p_{\tau_0} \neq 0\}\cap 
Gr(n-k,V)$. Here $ \{ p_{\tau_0 } \neq 0 \}$ is the basic affine subset of $\PP^{ {n\choose n-k} -1}$ defined as 
the set of points $x\in \PP^{ {n\choose n-k} -1}$ whose 1st coordinate is nonzero. We present $O_{\tau_0}$ 
more explicitly by using the description of $Gr(n-k,V)$ in terms of matrices: 
\begin{align}\label{A:standard form for U}
O_{\tau_0} = \left\{
[A] \in GL(n-k) \backslash Mat_{n-k,n}^0  :\ 
A=
\begin{pmatrix}
1 & 0  & \cdots & 0 & a_{11} & \cdots & a_{1k} \\
0 & 1  & \cdots & 0 & a_{21} & \cdots & a_{2k} \\
\vdots &  \vdots & \ddots & \vdots & \vdots & \ddots & \vdots  \\
0 & 0  & \cdots & 1 & a_{n-k,1} & \cdots & a_{n-k,k}
\end{pmatrix} 
\right\}
\end{align}

The natural action of $GL(n)$ on $(n-k) \times n$ matrices is given by $g\cdot A = A g^\top$, where $g \in GL(n)$ and $A\in Mat_{n-k,n}$. It induces an action on $Gr(n-k,V)$. Although it is not standard, we find it convenient to denote 
elements $[A]$ of $O_{\tau_0}$ in block form as in $[ id_{n-k} :\ A]$, where $A$ stands for the $(n-k) \times n$ 
matrix that defines $[A]$.

We record the following (notational) observation for future use.
\begin{Remark}\label{R:Future}
For $D\in GL(n-k)$, the equivalence class of the matrix $[D\cdot id_{n-k} : D \cdot A]$ in $Gr(n-k,V)$ 
defines the same subspace as $[id_{n-k} : A]$.
\end{Remark}

Let us denote by $w_{i,j}$ ($1\leq i \leq n-k$, $1\leq j \leq k$) the coordinate function $p_{\tau(i,j)}/p_{\tau_0}$, 
where $\tau(i,j)$ is the $n-k$-tuple that is obtained from $[n-k]$ by removing $j$ and adding $n-k+i$ instead. 
Thus $w_{i,j}([A])=a_{i,j}$. We are going to compute the action of $\varphi(z)$ on $O_{\tau_0}$ which will 
provide us with the action of $N_a$ on the $w_{i,j}$'s. These variables are the generating elements for the 
polynomial functions on the cotangent space at the origin $W_0 \leftrightarrow 
[id_{n-k} :\ \mathbf{0}_{n-k \times k} ]$ of $O_{\tau_0}$.

It is easy to see that 
\begin{align*}
\varphi(z) =\varphi_n(z) = 
\begin{pmatrix}
1 & z & \frac{z^2}{2!} & \cdots & \frac{z^{n-1}}{(n-1)!} \\ 
0 & 1 & z  & \cdots & \frac{z^{n-2}}{(n-2)!} \\ 
0 & 0 & 1  & \cdots & \frac{z^{n-3}}{(n-3)!} \\ 
\vdots & \vdots & \vdots & \ddots & \vdots \\
0 & 0 & 0 & \cdots & 1 
\end{pmatrix}.
\end{align*}
Let us denote the transpose of $\varphi(z)$ by $\varphi_n^\top(z)$ and write it in the block form as in 
\begin{align*}
\varphi_n^\top(z) = 
\begin{pmatrix} 
\varphi_{n-k}^\top (z) & 0 \\ 
Z & \varphi_k^\top(z)
\end{pmatrix},
\end{align*}
Note that the variable $z$ appears at the top right corner of $Z$ and at all other entries of $Z$ there are 
higher powers of $z$. Let us denote by $N_a^\top,N_{a,n-k}^\top$, and by $N_{a,k}^\top$ the nilpotent 
operators associated with $\varphi,\varphi_{n-k}$, and $\varphi_k$, respectively. 

Now, the action of $\varphi(z)$ on an element $[A] \in U_{\tau_0}$ is computed by block-matrix multiplication 
and it gives 
\begin{align*}
\varphi(z)\cdot [A] &= [ (id_{n-k} \ A ) \varphi_n^\top(z) ] \\
&= [ \varphi_{n-k}^\top(z) + AZ : \ A \varphi_k^\top(z) ].
\end{align*}
To bring this into its standard form as in (\ref{A:standard form for U}), we multiply it with the inverse of 
$\varphi_{n-k}^\top(z) + AZ$, which is always invertible:
\begin{align}\label{A:action of varphi}
\varphi (z) \cdot [A] = [ id_{n-k} :\  ( \varphi_{n-k}^\top(z) + AZ)^{-1}  A \varphi_k^\top(z) ]
\end{align}
Writing $\varphi_{n-k}^\top$ in the form $I_{n-k} + N$ where $N=N(z)$ is a (lower-triangular) nilpotent matrix, 
we have the formal expansion 
$$
( \varphi_{n-k}^\top(z) + AZ)^{-1} = I_{n-k} - ( N+AZ) + (N+AZ)^2 - (N+AZ)^3 + \cdots 
$$

We then can  compute the differential of the flow defined by (\ref{A:action of varphi}):
\begin{align*}
\frac{d\varphi (z)}{dz} \cdot [A] &= [ id_{n-k} :\  \frac{d}{dz} \cdot ( \varphi_{n-k}^\top(z) + AZ)^{-1}  
A \varphi_k^\top(z) ] \\
&= [ id_{n-k} :\  \left( \frac{d}{dz} \sum (-1)^j (N+AZ)^j \right)  A \varphi_k^\top(z)  + 
(\varphi_{n-k}^\top(z) + AZ)^{-1}  A  \frac{d}{dz} \varphi_k^\top(z) )
] \notag \\
\end{align*}

At $z=0$, this last expression simplifies to 
\begin{align}\label{A:A}
\left. \frac{d\varphi (z)}{dz}\right\vert_{z=0}\cdot [A]&= [id_{n-k}:\ -(N_{a,n-k}^\top+ A Z'|_{z=0})A + A N_{a,k}^\top],
\end{align}
where $Z'|_{z=0}$ is the matrix with single 1 at the top right corner and 0's elsewhere. It follows from (\ref{A:A}) 
that 
\begin{scriptsize}
\begin{align}
N_a \cdot A &= {- N_{a,n-k}A -  A Z'|_{z=0}A +  A N_{a,k}^{\top}} \notag  \\ 
&= -
\begin{pmatrix}
0 & \cdots & 0 \\
a_{1,1} & \cdots & a_{1,k} \\
\vdots & \ddots & \vdots \\
a_{n-k-1,1} & \cdots & a_{n-k-1,k} 
\end{pmatrix}
-
\begin{pmatrix}
a_{1,1} & \cdots & a_{1,k}   \\
a_{2,1} & \cdots & a_{2,k}  \\
\vdots & \ddots & \vdots \\
a_{n-k,1} & \cdots & a_{n-k,k} 
\end{pmatrix}
\begin{pmatrix}
a_{n-k,1} & \cdots & a_{n-k,k}   \\
0 & \cdots & 0  \\
\vdots & \ddots & \vdots  \\
0 & \cdots & 0 
\end{pmatrix}
+
\begin{pmatrix}
a_{1,2} & \cdots & a_{1,k} & 0  \\
a_{2,2} & \cdots & a_{2,k} & 0  \\
\vdots & \ddots & \vdots & \vdots \\
a_{n-k,2} & \cdots & a_{n-k,k} &0 \\
\end{pmatrix}	\notag \\
&= -
\begin{pmatrix}
0 & \cdots & 0 \\
a_{1,1} & \cdots & a_{1,k} \\
\vdots & \ddots & \vdots \\
a_{n-k-1,1} & \cdots & a_{n-k-1,k} 
\end{pmatrix}
-
\begin{pmatrix}
a_{1,1} a_{n-k,1} & \cdots & a_{1,1} a_{n-k,k}   \\
a_{2,1} a_{n-k,1} & \cdots & a_{2,1} a_{n-k,k}   \\
\vdots & \ddots & \vdots \\
a_{n-k,1} a_{n-k,1} & \cdots & a_{n-k,1} a_{n-k,k}   
\end{pmatrix}
+
\begin{pmatrix}
a_{1,2} & \cdots & a_{1,k} & 0  \\
a_{2,2} & \cdots & a_{2,k} & 0  \\
\vdots & \ddots & \vdots & \vdots \\
a_{n-k,2} & \cdots & a_{n-k,k} &0 
\end{pmatrix} \label{A:final relations}
\end{align}
\end{scriptsize}

Following Adriano Garsia's lead, let us agree on the following convenient notation: For two logical expressions 
$L_1$ and $L_2$, we put 
$$
\mt{\c{c}}_{L_1,L_2} =
\begin{cases}
0 & \text{ if } L_1= L_2 \\
1 & \text{ if } L_1 \neq L_2.
\end{cases}
$$
After reorganizing (\ref{A:final relations}) by using $\mt{\c{c}}$, we see that the action of $\varphi(z)$ on 
coordinate functions is given by 
\begin{align}\label{A:generators}
N_a (w_{i,j} ) &=  - \text{\c{c}}_{1,i}w_{i-1,j} - w_{i,1} w_{n-k, j} + \text{\c{c}}_{k+1,j+1} w_{i,j+1} 
\end{align}
for $j=1,\dots, k$ and $i=1,\dots, n-k$.

\vspace{.5cm}
This leads to the following theorem. We note that 
a version of this theorem was first obtained by Carrell and Lieberman in~\cite{MR0460734}.
\begin{Theorem}\label{T:Coh grassmannian}
The cohomology ring $H^*(Gr(n-k,V),\C)$, as a graded algebra, is the quotient ring $\C[ w_{i,j} ] / \tilde{I}$, 
where $\tilde{I}$ is the ideal generated by polynomials in (\ref{A:generators}).
The grading on $w_{i,j}$'s is imposed by the $\C^*$-action 
\begin{align}\label{T:as in}
\lambda(t) \cdot w_{i,j} = t^{-2(i-j) + 2(n-k)} w_{i,j}.
\end{align}
\end{Theorem}
\begin{proof}
In the light of the previous discussion, it suffices to prove that the $\mb{G}_m$-action is as in (\ref{T:as in}).
Let $[id_{n-k}: A]$ be an element of $O_{\tau_0}$, where 
$$
A= 
\begin{pmatrix}
1 & 0  & \cdots & 0 & a_{11} & \cdots & a_{1k} \\
0 & 1  & \cdots & 0 & a_{21} & \cdots & a_{2k} \\
\vdots &  \vdots & \ddots & \vdots & \vdots & \ddots & \vdots  \\
0 & 0  & \cdots & 1 & a_{n-k,1} & \cdots & a_{n-k,k}
\end{pmatrix}. 
$$
Then the action of $\lambda(t)=\mt{diag}(t^{n-1},\dots, t^{-n+1})$ on $[id_{n-k} : A]$ gives 
$$
[id_{n-k} : A] \lambda(t) 
= 
\begin{pmatrix}
t^{n-1} & 0  & \cdots & 0 & 	t^{n-(2((n-k)+1)-1)}	a_{11} & \cdots & t^{-n+1} a_{1k} \\
0 & t^{n-3}  & \cdots & 0 & 	t^{n-(2((n-k)+1)-1)} a_{21} & \cdots &  t^{-n+1}a_{2k} \\
\vdots &  \vdots & \ddots & \vdots & \vdots & \ddots & \vdots  \\
0 & 0  & \cdots & t^{n-(2(n-k)-1)} & 	t^{n-(2((n-k)+1)-1)} a_{n-k,1} & \cdots & t^{-n+1}a_{n-k,k}
\end{pmatrix}. 
$$
To write the equivalence class of this matrix in the form $[id_{n-k}: B]$, we multiply it on the left by the matrix 
$$
u(t):=
\begin{pmatrix}
t^{-(n-1)} & 0  & \cdots & 0 \\
0 & t^{-(n-3)}  & \cdots & 0 \\ 
\vdots & \vdots  & \ddots & \vdots \\ 
0 & 0 &  \cdots & t^{-(n-(2(n-k)-1))}
\end{pmatrix}
\in GL(n-k).
$$
It is now easy to see that the $(i,j)$-th entry of $u(t) [id_{n-k}: A] \lambda(t)$ is 
$$
t^{n-(2((n-k)+j)-1) -(   n-(2i-1))} a_{ij}= t^{ -2 (n-k) - 2(j-i) } a_{ij}.
$$ 
Note that to compute the action on functions $w_{i,j}$, we only need to apply $t \mapsto t^{-1}$. This completes 
the proof.
\end{proof}

\subsection{$\C^*$-equivariant cohomology ring of $X$ as a $B(2)$-variety.}
\label{S:Equivariant Regular grassmannian}


Our goal in this section is to describe the $\C^*$-equivariant cohomology ring of $X=Gr(n-k,V)$. Let us stress 
once more that the novelty here is the method itself rather than the final result. There are many articles in the 
literature where the $T$-equivariant cohomology ring $H_T^*(Gr(n-k,V))$ is computed. Here, $T$ is a torus 
acting on $Gr(n-k,V)$. See, for example, the paper~\cite{MR1997946} of Knutson and Tao, where the 
structure constants of $H_T^*(Gr(n-k,V))$ are computed via combinatorial objects called ``puzzles'' which were  
introduced therein.

We start with reminding the reader how a $K$-equivariant cohomology ring of a variety $X$ is defined. Here, 
$K$ is a Lie group with an algebraic action on $X$. Let $EK\rightarrow EB$ denote the universal principal 
bundle for $K$, which means that $EK$ is a contractible space with a free action of $K$ and $EB$ is the 
classifying space for $K$-principal bundles on $X$. The $K$-equivariant cohomology, $H_K^*(X)$ of $X$ 
is, by definition, the ordinary cohomology of the space $EK\times_K X$ obtained from $EK\times X$ by 
taking the quotient by diagonal action of $K$.

We will now explain a result of Brion and Carrell from~\cite{MR2043405} which is helpful for computing the 
$\C^*$-equivariant cohomology ring of a regular $B(2)$-variety. In fact, this result applies to pairs $(Y,X)$, 
where $i_Y: Y\hookrightarrow X$ is a $B(2)$-stable subvariety of $X$ and the restriction map $i_Y^*:H^*(X) 
\rightarrow H^*(Y)$ is surjective. Such a subvariety is termed as a "principal subvariety of $X$'' 
in~\cite{MR2452602}.

Let $v$ denote the affine coordinate on $\C=\PP^1- \{ (1,0) \}$. In~\cite{MR2043405}, it is shown that if $Y
\subset X$ is a principal subvariety, then there exists a $\mb{G}_m$-stable affine curve $\mc{Z}_X$ in $X
\times \PP^1$ and a graded $\C$-algebra isomorphism $\rho_X : H^*_{\mb{G}_m} (X) \rightarrow \C[ 
\mc{Z}_X ]$. Furthermore, if $\mc{Z}_Y$ denotes the reduced affine curve $\mc{Z}_X\cap (Y\times \C)$, 
then there is an additional isomorphism $\rho_Y : H^*_{\mb{G}_m} ( Y) \rightarrow \C[\mc{Z}_Y]$ which 
makes the diagram in Figure~\ref{F:Commutative} commutative.
\begin{figure}[htp]
\centering
\begin{tikzpicture}[scale=1]

\begin{scope}
\node (a1) at (-2,2) {$H^*_{\mb{G}_m} ( X) $} ;
\node (a2) at (-2,0) {$H^*_{\mb{G}_m} ( Y) $} ;

\node (b1) at (2,2) {$ \C[ \mc{Z}_X ]$} ;
\node (b2) at (2,0) {$ \C[ \mc{Z}_Y ]$} ;

\node at (0,2.25) {$\rho_X$};
\node at (0,0.25) {$\rho_Y$};

\node at (-2.35,1) {$i_Y^*$};
\node at (2.35,1) {$\overline{i_Y^*}$};

\draw[->, thick] (a1) to (a2);
\draw[->, thick] (a1) to (b1);
\draw[->, thick] (b1) to (b2);
\draw[->, thick] (a2) to (b2);
\end{scope}
\end{tikzpicture}
\caption{Cohomology of a principal subvariety.}
\label{F:Commutative}
\end{figure}
Moreover, in the same figure, the horizontal maps are $\C[v]$-module maps for the standard $\C[v]$-module 
structure on $H^*_{\mb{G}_m} (X)$ and $H^*_{\mb{G}_m} (Y)$ and the $\C[v]$-module structure on $\C[
\mc{Z}_X ]$ and $ \C[ \mc{Z}_Y]$ induced by the second projection.
Unfortunately, as we will show in the sequel, $Hilb_k(\PP^2)$ is not a principal subvariety in the grassmannian, 
so we must use other methods for computing its $\C^*$-equivariant cohomlogy ring.

Recall our notation from Subsection~\ref{SS:B2}; $\phi: \mb{G}_a \rightarrow B(2)$ and $\lambda:\mb{G}_m
\rightarrow B(2)$ denote, respectively, the one-parameter subgroups determined by the the unipotent radical 
and the maximal torus of $B(2)$. $X$ is a regular $B(2)$-variety and $s_0$ denotes the unique 
$\mb{G}_a$-fixed point in $X$, $X_0$ denotes the minus-cell at $s_0$. Let $\mc{W}$ denote the vector field 
that is obtained by differentiating $\lambda(t)$ at $t=1$ and let $\mc{V}$ denote the vector field $\frac{d}{dz} 
\varphi \vert_{z=0}$ (so, $\mc{V}=N_a$).
\begin{Lemma}[Proposition 2,\cite{MR2043405}]\label{L:BC}
The scheme $\mc{Z}_X$ is contained in $X_0 \times \A^1$ as a $\mb{G}_m$-curve. If $s_0,s_1,\dots, s_r$ is 
the list of $\mb{G}_m$-fixed points on $X$, then irreducible components of $\mc{Z}_X$ are of the form 
$$
\mc{Z}_j= \{ (\varphi(z) \cdot s_j, z^{-1}):\ z\in \C^* \}  \cup \{ (s_0,0)\} \cong \PP^1- pt,\ \text{ for } 
j=1,\dots, r.
$$
Furthermore, any two such component meet only at $(s_0,0)$. The ideal of $\mc{Z}_X$ in $\C[X_0\times \A^1] 
= \C[x_1,\dots, x_n,v]$ is generated by 
$$
2\mc{W}(x_1) - \mc{V}(x_1),\dots,2\mc{W}(x_n) - \mc{V}(x_n).
$$
These polynomials form a regular sequence in $\C[x_1,\dots, x_n,v]$ and the degree of each $2\mc{W}(x_i) - 
\mc{V}(x_i)$ equals $-a_i+2$. 
\end{Lemma}

Now we are ready to record a description of the $\mb{G}_m$-equivariant cohomology ring of $Gr(n-k,V)$. 
We use the notation of Theorem~\ref{T:Coh grassmannian}.
\begin{Theorem}\label{T:Eq coh grassmannian}
The $\mb{G}_m$-equivariant cohomology ring of $Gr(n-k,V)$ is isomorphic to the quotient (coordinate) 
ring $\C[\mc{Z}_X] = \C[w_{i,j},v]/\widetilde{I}_v$, where $\widetilde{I}_v$ is the ideal generated by 
\begin{align}\label{A:generators of I}
 (2(n-k)-2(i-j)) v w_{i,j}  - 2( -\text{\c{c}}_{1,i}w_{i-1,j} - w_{i,1} w_{n-k, j} + \text{\c{c}}_{k+1,j+1} w_{i,j+1} )
\end{align}
for $j=1,\dots, k$ and $i=1,\dots, n-k$.
\end{Theorem}
\begin{proof}
We already computed in Theorem~\ref{T:Coh grassmannian} the action of $\mc{V}$ on the coordinate functions 
$w_{i,j}$ of the minus-cell $O_{\tau_0}\subset Gr(n-k,V)$. Moreover, we computed the action of $\lambda(t)$ on 
$w_{i,j}$'s. Since $\frac{d}{dt}\lambda \vert_{t=1}(w_{i,j})= (2(n-k)-2(i-j))w_{i,j}$, the proof follows from 
Lemma~\ref{L:BC} and the aforementioned results of Brion and Carrell.
\end{proof}

\section{Torus action}\label{S:GaGm variety}

The Hilbert scheme of $k$ points on a variety $M$, denoted by $Hilb_k(M)$ is, by definition, the scheme that 
represents the functor from the collection of all subschemes of the plane with constant Hilbert polynomial 
$p(x)=k$~\cite{MR1611822}. It has been  known for sometime that if $M=\PP^n$ is a projective space, then $Hilb_k(M)$ 
is connected (see~\cite{MR0213368}) but it is not always nonsingular or irreducible except when $n=2$. 
See~\cite{MR0335512}. For an introduction to the Hilbert scheme of points with an eye towards combinatorial 
commutative algebra, we recommend~\cite{MR2110098}. 

We now show that, in general, the Hilbert scheme of points on a surface is not a regular $B(2)$-variety. To this 
end we consider the Poincar\'e polynomial of $Hilb_3(\PP^2)$, which is 
$$
g_3(t) := 1+ 2t^2 + 5t^4 + 6t^6+5t^8 + 2t^{10} + t^{12}.
$$
(See~\cite[Table 1]{MR870732}.) It can be checked by a computer that $g_3(t)$ has a root that is not a root of 
unity, hence, by Remark~\ref{A:roots of unity}, $Hilb_3(\PP^2)$ is not a regular $(\mb{G}_a,\mb{G}_m)$-variety. 
Although we checked only a few other cases, we anticipate that for all $k\geq 3$, the Poincar\'e polynomial of 
$Hilb_k(\PP^2)$ has a root which is not a root of unity. For $k=2$, we can make the following remark.

\begin{Remark}\label{R:fails}
There is a natural action of $GL(3)$ on $Hilb_k(\PP^2)$ which is induced from its action on polynomials: Let 
$g=\begin{pmatrix}
g_{11} & g_{12} & g_{13} \\
g_{21} & g_{22} & g_{23} \\
g_{31} & g_{32} & g_{33} 
\end{pmatrix}
$ 
be an element from $GL(3)$. Then $g$ acts on variables by $g\cdot X_i = \sum_{j=1}^3 g_{i j} X_j$ for $i=1,2,3$. 
In particular, the additive one parameter subgroup $\mb{G}_a=\C$, viewed as a subgroup of $GL(3)$ acts on the 
variables by 
\begin{align*}
\varphi(z) \cdot 
\begin{cases}
X_0 & \mapsto X_0 + z X_1 + z^2 X_2  \\
X_1 & \mapsto X_1 + z X_2  \\
X_2 & \mapsto  X_2. 
\end{cases}
\end{align*}
The resulting action on a degree $k$ monomial is given by 
\begin{align}\label{A:other terms}
\varphi(z) \cdot X_0^a X_1^b X_2^c = X_0^a X_1^b X_2^c + z( X_0^{a}X_1^{b-1} X_2^{c+1}  + 
X_0^{a-1} X_1^{b+1} X_2^c) + \text{ other terms }
\end{align}
where we assume $a,b \geq 1$. (Of course, if $a=0$ and $b>0$, then the coefficient of $z$ is $X_1^{b-1} 
X_2^{c+1}$. If $a>0$ and $b=0$, then it is just $X_0^{a-1}X_1X_2^c$.) Unfortunately, the nilpotent vector 
field that is obtained by differentiating (\ref{A:other terms}) is not regular, that is to say its Jordan form has 
more than one block. Indeed, when $k=2$, the matrix of the nilpotent operator with respect to basis 
$\{X_2^2, X_2X_1,X_2X_0, X_1^2, X_1X_0, X_0^2 \}$ is 
\begin{align*}
 \left.\frac{ d \varphi }{ d z } \right\vert_{z=0} =
\begin{pmatrix}
0 & 1 & 0 & 0 & 0 & 0 \\
0 & 0 & 1 & 2 & 0 & 0 \\
0 & 0 & 0 & 0 & 1 & 0 \\
0 & 0 & 0 & 0 & 1 & 0 \\
0 & 0 & 0 & 0 & 0 & 2 \\
0 & 0 & 0 & 0 & 0 & 0 \\
\end{pmatrix}.
\end{align*}
The Jordan form of this matrix is given by  
\begin{align*}
\begin{pmatrix}
0 & 1 & 0 & 0 & 0 & 0 \\
0 & 0 & 1 & 0 & 0 & 0 \\
0 & 0 & 0 & 1 & 0 & 0 \\
0 & 0 & 0 & 0 & 1 & 0 \\
0 & 0 & 0 & 0 & 0 & 0 \\
0 & 0 & 0 & 0 & 0 & 0 \\
\end{pmatrix}.
\end{align*}
The fixed locus of the resulting $\mb{G}_a$-action on the Grassmann of 4 dimensional subspaces of 
$$
\C \{X_2^2, X_2X_1,X_2X_0, X_1^2, X_1X_0, X_0^2 \}
$$ 
is isomorphic to the projective line $\PP^1$. Consequently, the resulting action on the Hilbert scheme of 2 points 
is not regular. Similar to the case of projective plane, for $\PP^n$, there is a natural $GL(n)$-, hence a 
$\mb{G}_a$-action on $Hilb_k(\PP^{n})$. In~\cite{KO}, by analyzing the corresponding 
$(\mb{G}_a,\mb{G}_m)$-graph of $Hilb_k(\PP^{n})$, Ki{\c{s}}isel and \"Ozkan gave a new proof of the 
connectedness of $Hilb_k(\PP^{n})$. The edges of the graph correspond to the projective lines that are fixed by 
the action of $\mb{G}_a$.
\end{Remark}

Next, we give a very brief account of elementary G\"obner basis theory that will be helpful for explaining the 
Gotzmann embedding and for the construction of our $\C^*$-action. Let $V$ be a vector space, $\{x_1',
\dots, x_n'\}$  be a basis of $V$, and let $S$ denote the symmetric algebra of $V$. Thus, $S$ is isomorphic to the 
polynomial ring $\C[x_1,\dots, x_n]$, where $\{x_1,\dots, x_n\}$ is dual basis to $\{x_1',\dots, x_n'\}$. We set 
$x^J=x_1^{j_1}\cdots x_n^{j_n}$ whenever $J=(j_1,\dots, j_n) \in \N^n$. Let $|J|$ denote the sum $j_1+\cdots 
+ j_n$.  A total order on monomials of degree $d$ is called a multiplicative order if the following properties hold 
true:
\begin{enumerate}
\item $x_1 > \cdots > x_n$ and 
\item for all $J,J',K\in \N^n$, if $x^J > x^{J'}$, then $x^K x^{J} > x^K x^{J'}$.
\end{enumerate}
We extend the multiplicative order to a monomial order by insisting on the ``usual'' requirement that $x^K > x^J$ 
if and only if $d_1 > d_2$ for all monomials $x^K$ and $x^J$ of total degree $d_1$ and $d_2$, respectively. The 
{\em lex ordering} is the order on $S$ such that, for multi-indices $K=(k_1,\dots, k_n)$ and $J=(j_1,\dots, j_n)$ 
with $|K|= |J|$, $x^K > x^J$ if and only if there exists $s\in \{1,\dots, n\}$ such that $k_s > j_s$ and $k_i=j_i$ for 
$i\in \{1,\dots, s-1\}$.


Let $I$ be a homogenous ideal in $S$. The initial ideal of $I$, denoted by $in(I)$, is the ideal generated by all 
initial monomials of the elements of $I$. For each monomial $x^J\in in(I)$, there is a homogenous polynomial 
$f_J\in I$ such that $in(f_J)= x^J$. A subset in $I$ consisting of elements $f_J\in I$ whose initial monomials 
form a basis for $in (I)_d$ is called a standard basis for $I_d$. We borrow the following fact from
~\cite[Proposition 1.11]{MR2640511}:
\begin{Lemma}\label{L:basis}
For any homogenous ideal $I=\oplus_{d\geq 1} I_d$, which is graded by degree of its elements, a standard 
basis for the $d$-th graded piece is in fact a basis for $I_d$. In particular, the vector spaces $I_d$ and 
$in(I)_d$ have the same dimension. Moreover $I$ and $in(I)$ have the same Hilbert function. 
\end{Lemma}

We will apply these considerations to the ideals in $Hilb_k(\PP^2)$, so, let $X_0,X_1,X_2$ be a system of 
coordinates on the plane $\PP^2$. The polynomial ring $R=\C[X_0,X_1,X_2]$ is graded by degree, 
$R=\oplus_{d\geq 0} R_d$, and the points in $Hilb_k(\PP^2)$ are homogeneous ideals of colength $k$. 
Therefore, if $I\in Hilb_k(\PP^2)$, then $I$ is a homogeneous ideal with Hilbert polynomial $Q(t)={t+ 2 
\choose 2} - k$. The ``Gotzmann number'' for the bound of regularity for such ideals is $k$. 
See \cite[Lemma 2.9, pg 65]{MR0480478}. Set $P_k:=Q(k)=n-k$. Let $\ms{A}_k$ denote the set of 
monomials of total degree $k$. It is clear that $\ms{A}_k$ is a basis for $R_k$, hence the dimension of 
$R_k$ is $n=n_k:=|\ms{A}_k|= {k+2 \choose k}$. It follows from Gotzmann's work~\cite{MR0480478} 
that the map 
\begin{align*}
Hilb_k(\PP^2) &\longrightarrow Gr( P_k, R_k)\\
I& \longmapsto I \cap R_k
\end{align*}
is a closed embedding. 

Our preferred monomial ordering on $R$ is the lexicographic ordering with $X_0 > X_1 > X_2$. We use it to 
order $\ms{A}_d$ as follows:
\begin{align}\label{A:revlex}
{\ms{A}_d}:= \{ X_2^d, \ X_2^{d-1}X_1, \dots,  X_1^d, X_2^{d-1}X_0, \ X_2^{d-2} X_1X_0, \dots, 
X_1^{d-1}X_0, \dots, X_0^d \}.
\end{align}
Let $e_1,\dots, e_d$ denote the elements of ${\ms{A}_d}$ in the increasing order as above. Let $\lambda_0,
\lambda_1,\lambda_2$, and $g>1$ be positive integers. Define $\lambda: \C^*\times R_1 \rightarrow R_1$ 
by
\begin{align}\label{A:T action}
\lambda(t) \cdot X_i  &= t^{ g^{\lambda_i} } X_i\ \text{  for } i=0,1,2.
\end{align}
$\lambda$ extends to give a $\C^*$-action on each component $R_d$. For a monomial $X=X_0^{a_0}X_1^{a_1}X_2^{a_2} 
\in R_d$, we set $\lambda(t)\cdot X =t^{\alpha}X$, where $\alpha = a_0 g^{\lambda_0} + a_1 g^{\lambda_1} + 
a_2 g^{\lambda_2}$. It is clear now that $\lambda$ operates on homogenous ideals of $R$ as well as on the 
grassmannian of $r$ dimensional subspaces of $R_d$, for each degree $d \geq 1$ and $1\leq r \leq \dim R_d$. 
In relation with the Gotzmann embedding of $Hilb_k(\PP^2)$ into $Gr(n-k,R_k)$, we insist that $\lambda$ 
satisfies the following conditions:
\begin{enumerate}
\item $g^{\lambda_0} > a g^{\lambda_1} + b g^{\lambda_2}$ for any nonnegative integers $a$ and $b$ such 
that $a+b=k$;
\item $g^{\lambda_1} > cg^{\lambda_2}$ for any nonnegative integer $c$ such that $0\leq c \leq k$.
\end{enumerate}

\begin{Lemma}\label{L:Taction}
Let $X$ and $Y$ denote the Grassmann variety $Gr(n-k,R_k)$ and the Hilbert scheme $Hilb_k(\PP^2)$, 
respectively. The $\C^*$-action (\ref{A:T action}) satisfying the two additional conditions of the previous 
paragraph acts on both of $X$ and $Y$. Moreover, $W\in X^\lambda$ (the fixed point set of $\lambda$) 
if and only if $W$ is a coordinate subspace, that is to say, $W$ is spanned by a subset with cardinality 
$n-k$ of $\ms{A}_k$. A point $I\in Y$ is invariant under the action if and only if $I$ is a monomial ideal, 
so $I \cap R_k=W$ for some $W\in X^\lambda$.
\end{Lemma}
\begin{proof}
The fact that $\C^*$ acts (via $\lambda$) on both of $X$ and $Y$ is proven in~\cite[Section 2]{MR968319}. 
The additional conditions on $\lambda$ are required so that the action on $X$ has distinct weights on 
$\ms{A}_k$, hence it is a regular (in the sense of \cite[Section 12]{Borel}) one-parameter subgroup action. 
It is well known that such a $\C^*$ action has finitely many fixed points on the grassmannian as specified 
in the latter statement of the lemma. 
\end{proof}

\section{Symmetric functions and cohomology}\label{S:Carrell}
Our goal in this section is to analyze the ring homomorphism between the associated graded rings of the rings 
of functions on the fixed point schemes $Z\subset X=Gr(n-k,R_k)$ and $Z'\subset Y=Hilb_k(\PP^2)$. Recall 
that the $\C^*$-action on $X$ is induced by the one-parameter subgroup $\lambda : \C^*\rightarrow GL(R_k)$. 
We know that $\lambda$ is diagonalizable with respect to the basis $\ms{A}_k$. Let us denote its eigenvalues 
by 
$$
\lambda_1=\lambda_1(t),\dots, \lambda_n=\lambda_n(t),
$$ 
which are distinct from each other. See the proof of Lemma~\ref{L:Taction}. Also from Lemma~\ref{L:Taction},  
we see that both as a scheme and as a set $Z$ consists of ${n \choose k}$ reduced points $W_I \in X$ $(I=
(1\leq i_1 < \dots < i_{n-k} \leq n))$. Here, $W_I$ is the subspace spanned by $\{e_{i_1},\dots, e_{i_{n-k}}\}$ 
($e_i \in \ms{A}_k$).

The evaluation map $p\mapsto p(\lambda)$ defined on polynomials in $\C[x_1,\dots, x_n]$ gives a natural 
morphism $\rho: \C[x_1,\dots,x_n] \rightarrow H^0(Z)$. Let $\sigma_1,\dots, \sigma_n$ denote the 
elementary symmetric polynomials in variables $x_1,\dots, x_n$ and let $H$ the subspace in $\C[\sigma_1,
\dots,\sigma_{n-k}]$ spanned by the $m$-fold ($1\leq m \leq k$) products of $\sigma_1,\dots, \sigma_{n-k}$ 
along with the constant polynomial 1. It is shown in \cite[Lemma 1]{MR0469928} that the image of $H$ 
under $\rho$ is isomorphic to $H^0(Z)$. In fact, much more is shown to be true 
\cite[Theorem 1]{MR0469928}. Let $p$ be a number such that $0\leq p \leq k(n-k)$ and let $B_p$ denote 
the set of Schur polynomials which satisfy
\begin{enumerate}
\item $\mu$ is a partition of $r$ with $0\leq r \leq p$;
\item $\mu$ has at most $n-k$ parts;
\item the largest part of $\mu$ is $k$.
\end{enumerate}
Then $B_p$ is an additive basis of $H^0(Z)\cap F_{-p}$. Consequently, the images in $F_{-p}/F_{-p+1}$ of 
the Schur polynomials $\{ s_\mu\}_{\mu \in B_p}$ form a basis for $H^{p}(X,\C)$, where $X=Gr(n-k,R_k)$.

Now we are ready to give a presentation of the cohomology ring of the Hilbert scheme of points by using the 
above description of the cohomology ring of the Grassmann variety. 
\begin{Theorem}
The images under $\phi'$ of the residue classes of the Schur polynomials $s_\mu$, where $\mu\in B_p$ 
and $0\leq p \leq 2k$ form a basis for $H^{p}(Y,\C)$. 
\end{Theorem}

\begin{proof}
In the light of Remark~\ref{R:importance of}, the importance of Lemma~\ref{L:Taction} becomes clear once we 
use it with \cite[Theorem 3]{MR827353}. Let $Z$ denote the fixed point set of $\lambda$ on $X:=Gr(n-k,R_k)$ 
and let $Z'$ denote $Z\cap Y = Z\cap Hilb_k(\PP^2)$. Since $Z$ is finite and 
the one-parameter subgroup 
$\lambda$ is regular, the connected components of $Z$ are irreducible. It follows that the inclusion 
$Z'\hookrightarrow Z$ induces a surjection on the cohomology level, hence we get the commuting diagram 
(\ref{F:simplified}) where $X=Gr(n-k,R_k)$ and $Y=Hilb_k(\PP^2)$. Since $Y$ is nonsingular and all odd Betti 
numbers of $Y$ are zero, it follows from \cite[Theorem 1]{MR827353} that the finiteness of $Z'$ implies 
$\psi:\textbf{gr}\ H^0 ( Z') \rightarrow H^*(Y)$ is an isomorphism. Therefore, the arrow on the right hand side 
of (\ref{F:simplified}) is a surjective graded algebra homomorphism. To fully describe the structure of $H^*(Y)$ 
it remains to understand the surjective homomorphism $\phi':\textbf{gr}\ H^0(Z)\rightarrow\textbf{gr}\ H^0(Z')$. 
But the canonical map $\phi$ induces a filtration 
$$
H^0(Z') = G_{-2k} \supset G_{-2k+1}\supset \cdots \supset G_0=0
$$ 
with $G_{-p}= \phi( F_{-p})$ ($0\leq p \leq 2k=\dim Y$). Therefore, by passing to the associated graded ring, 
our claim follows from the above discussion.
\end{proof}

\section{Equivariant cohomology revisited}\label{S:ECoh}
Our goal in this section is to give a presentation of the equivariant cohomology ring of $Y=Hilb_k(\PP^2)$ by fine-tuning the results of Aky{\i}ld{\i}z, Carrell, and Lieberman from \cite{MR827353} in the 
equivariant setting. We start with reviewing some well known facts about vector bundles on the Hilbert 
scheme.

\subsection{Tautological bundles.}
As before, we denote the grassmannian $Gr(n-k,R_k)$ by $X$ and $\lambda$ stands for the one-parameter 
subgroup (\ref{A:T action}). Let $V$ denote the vector space $R_k$, $p=(a_0,a_1,a_2)$ be a point in $\C^3$ 
and let $\widetilde{\psi} (p)$ denote 
$$
\widetilde{\psi}(p)=(a_2^k , a_2^{k-1}a_1, \dots, a_0^k) \in V^*,
$$
where the ordering of the entries of $\widetilde{\psi}(p)$ are as in (\ref{A:revlex}). Passing to quotient, 
$\widetilde{\psi}$ defines an embedding $[p]\mapsto [\psi(p)]$ of $\PP^2$ into $\PP^{n-1}=\PP(V^*)$. 
The reason for we are considering $\PP^2\hookrightarrow \PP(V^*)$ is because the ``universal family'' 
on $Y$ is directly related to the tautological (universal) vector bundle of $X$. By definition, the universal 
family on $Y$, denoted by $\mc{F}_k$, is the incidence variety 
$$
\mc{F}_k = \{ ( I,p) :\ I\in Y \text{ and } \ p \in V(I) \}.
$$
The push down by the canonical projection $p_1 : \mc{F}_k \rightarrow Y$ of the structure sheaf of 
$\mc{F}_k$ gives a rank $k$ vector bundle on $Y$. This is easy to check on the locus consisting of 
radical ideals. For reasons to become clear soon, let us consider $Gr(k,V^*)$. The tautological 
bundle on $Gr(k,V^*)$, denoted by $\ms{S}$, is the incidence variety whose underlying set of points 
is 
$$
\mc{S} = \{ ( W,p) :\ W\in Gr(k,V^*) \text{ and } \ p \in W \}.
$$
Note that there is a natural duality isomorphism between $Gr(k,V^*)$ and $Gr(n-k,V)$. Note also that 
$Gr(n-k,V)$ is the same variety as $Gr(n-k-1,\PP(V))$, the grassmannian of $n-k-1$ dimensional 
projective subspaces in $\PP(V)$. Similarly, $Gr(k,V^*)$ is the same variety as $Gr({k-1},\PP(V^*))$.

Now, viewing $Y$ as an embedded subvariety of $X=Gr(n-k-1,\PP(V))$, we see that the pull back of $\mc{S}$ 
from $X$ to $Y$ is isomorphic $\mc{F}_k$. The $\C^*$-action on $R_k$ gives rise to a torus action on 
$\mc{S}$. Since the embedding $Y\hookrightarrow X$ is $\lambda$-equivariant, $\mc{O}_{\mc{F}_k}$ 
is a $\lambda$-equivariant bundle on $Y$. Clearly, the first projection $p_1 : \mc{F}_k \rightarrow Y$ is 
$\lambda$-equivariant as well. In~\cite{MR1228610}, it is shown by Ellingsrud and Str{\o}mme that the 
cohomology ring $H^*(Y)$ is generated by the Chern classes of the following rank $k$ bundles:  
$$
{p_1}_*(\mc{O}_{\mc{F}_k} \otimes {p_2}^* \mc{O}_{\PP^2} (-j) )\qquad \text{for } j=1,2,3,
$$
where $p_2$ is the second projection from $\mc{F}_k$ onto $\PP^2$. Since $\mc{O}_{\mc{F}_k}$ is 
$\lambda$-equivariant, the tensor products of $\mc{O}_{\mc{F}_k} \otimes {p_2}^* \mc{O}_{\PP^2} (-j)$ 
with the line bundles $\mc{O}_{\PP^2}(-j)$, $j=1,2,3$ is still $\lambda$-equivariant. In particular, the 
bundles ${p_1}_*(\mc{O}_{\mc{F}_k} \otimes {p_2}^* \mc{O}_{\PP^2} (-j))$'s are all 
$\lambda$-equivariant.

\subsection{Equivariant Chern classes.}

Let $K$ be a Lie group and let $X$ be a $K$-space. A vector bundle $\mc{E}$ on $X$ is called $K$-linearized 
if the $K$ action on $X$ lifts to $\mc{E}$ such that each $g\in K$ defines a linear map from $\mc{E}_x$ to 
$\mc{E}_{g\cdot x}$, where $\mc{E}_x$ and $\mc{E}_{g\cdot x}$ are the fibers of $\mc{E}\rightarrow X$ on 
$x$ and $g\cdot x$, respectively.  In relation with $K$-linearized vector bundles, we have the notion of 
equivariant Chern classes: the $k$-th $K$-equivariant Chern class $c_k^K(\mc{E}) \in H^k_K(X)$ is the $k$-th 
Chern class of the vector bundle 
$$
\mc{E} \times EK \rightarrow (X\times EK)/K.
$$
Note that if $x\in X^K$ is a fixed point of the action of $K$, then the restriction of $c_k^K(\mc{E})_x$ is contained 
in $H^*_K(x)$, which is isomorphic to the coordinate ring of $K$. From now on we focus on the group $K=\C^*$ 
which acts on $X$ and $Y$ by $\lambda$. Note that $H^*_K(pt) = \C[ v,v^{-1}]$, where $v$ is a variable. It is clear 
that the tautological bundle as well as the ``universal quotient bundle'' on $X=Gr(n-k,V)$ and the universal family 
on $Y=Hilb_k(\PP^2)$ are $K$-linearized. Since $H^*(X)\rightarrow H^*(Y)$ is surjective, (from Leray spectral 
sequence),  we see that $H^*_K(X)\rightarrow H^*_K(Y)$ is surjective also. Next, we will use a result of Brion 
\cite{MR1466694} for computing these rings. This is the approach taken by Evain in \cite{MR2323683} for 
computing the equivariant Chow rings of $X$ and $Y$.

Let $X$ denote a nonsingular projective variety on which an algebraic torus $T$ acts. Let $M$ denote the 
character group of $T$ and denote by $S$ the symmetric algebra over $\Z$ of the abelian group $M$. The 
following result is a simplification by Brion of Edidin and Graham's localization theorem for equivariant 
Chow rings \cite{MR1623412}. Since we are working with cohomology, we modified the statement 
accordingly.

\begin{Theorem}(\cite[Corollary 2.3.2]{MR1466694})\label{T:BGE}
Let $i : X^T \hookrightarrow X$ denote the inclusion of fixed point subscheme $X^T$ into $X$. Then $S$-linear 
map $i^* : H^*_T(X^T,\Q) \rightarrow H^*_T(X,\Q)$ becomes an isomorphism after inverting all nonzero elements 
of $M$. 
\end{Theorem}
\begin{Remark}\label{R:quotient}
Let us point out also that the ordinary cohomology $H^*(X,\Q)$ is the quotient of $H^*_T(X,\Q)$ by its subgroup 
$MH^*_T(X,\Q)$, see~\cite[Corollary 2.3.1]{MR1466694}.
\end{Remark}

In our situation, $T=\C^*$, $M\cong \Z$, and $S=\Z[M]\cong \Z[v]$, where $v$ is a variable that we use in place 
of the primitive character of $T$. The action of $\C^*$ is as defined in (\ref{A:T action}).

\begin{Theorem}\label{T:Eq coh Hilb}
Let $Z$ and $Z'$ denote, as before, the fixed point schemes in $X=Gr(n-k,R_k)$ and $Y=Hilb_k(\PP^2)$,
respectively. Then the $\C^*$-equivariant cohomology rings of $X$ and $Y$ are 
\begin{align}\label{A:Eq coh Hilb}
H^*_{\C^*} (X,\Q) \cong \bigoplus_{ W\in Z} \Q [v , v^{-1}] \ \text{ and } \ H^*_{\C^*} (Y,\Q) \cong 
\bigoplus_{ W\in Z'} \Q [v , v^{-1}].
\end{align}
Furthermore, the following diagram where all arrows are surjective $\C[v]$-algebra homomorphisms commutes:
\begin{figure}[htp]
\centering
\begin{tikzpicture}[scale=1]

\begin{scope}
\node (a1) at (-2,2) {$H^*_{\C^*} (X;\C)$} ;
\node (a2) at (-2,0) {$H^* (X)$} ;
\node (b1) at (2,2) {$H^*_{\C^*} (Y;\C)$} ;
\node (b2) at (2,0) {$H^*(Y)$} ;

\draw[->, thick] (a1) to (a2);
\draw[->, thick] (a1) to (b1);
\draw[->, thick] (b1) to (b2);
\draw[->, thick] (a2) to (b2);
\end{scope}
\end{tikzpicture}
\end{figure}
\end{Theorem}

\begin{proof}

The $T$-equivariant cohomology of a reduced zero dimensional scheme $Z=\{x_1,\dots, x_m\}$ is $H^*_T(Z) 
= \oplus_{x_i\in Z} H^*_T(pt)$ and furthermore $H^*_T(pt)= \Z[v,v^{-1}]$ for $T=\C^*$. Therefore, the first part 
of the theorem follows from Theorem~\ref{T:BGE}. The claim about the commutative diagram follows from the 
first part, Remark~\ref{R:quotient}, and from the discussion prior to Theorem~\ref{T:BGE}.
\end{proof}

\bibliography{References}
\bibliographystyle{plain}
\end{document}